\documentclass[12pt,twoside]{article}

\usepackage{amssymb,amsthm}
\usepackage{amsmath}
\usepackage{color}
\usepackage{esint}
\usepackage{mathabx}
\usepackage{MnSymbol}
\usepackage{fancyhdr}
\usepackage{times}

\usepackage[latin1]{inputenc}

\usepackage{comment}
\usepackage{url}
\usepackage{xcolor}
\usepackage{adjustbox}
\usepackage{hyperref}
\usepackage{mathtools}
\usepackage{authblk}

\mathtoolsset{showonlyrefs}

\newtheorem{thm}{Theorem}[section]

\newtheorem{lem}[thm]{Lemma}

\newtheorem{defi}[thm]{Definition}

\theoremstyle{remark}
\newtheorem{comentario}{Remark}

\makeatletter
\newcommand{\labitem}[2]{%
\def\@itemlabel{\textbf{#1}}
\item
\def\@currentlabel{#1}\label{#2}}
\makeatother

\date{}

\newcommand{\orlnor}{\|_{L^{\Phi}}}

\newcommand{\lphi}{L^{\Phi}}
\newcommand{\lpsi}{L^{\Psi}}
\newcommand{\ephi}{E^{\Phi}}
\newcommand{\claseor}{C^{\Phi}}
\newcommand{\wphi}{W^{1}\lphi}
\newcommand{\wphiet}{W^{1}\ephi_T}

\newcommand{\sobnor}{\|_{W^{1}\lphi}}
\newcommand{\domi}{\mathcal{E}^{\Phi}_d(\lambda)}
\renewcommand{\b}[1]{\boldsymbol{#1}}
\newcommand{\rr}{\mathbb{R}}

\newcommand{\ccdot}{\b{\cdot}}
\renewcommand{\leq}{\leqslant}
\renewcommand{\geq}{\geqslant}
\newcommand{\epsi}{E^{\Psi}}

\newcounter{example}

\title{Periodic solutions of
Euler-Lagrange equations with sublinear potentials   in an Orlicz-Sobolev space setting}

\author{Sonia Acinas\thanks{Partially supported by a  UNSL grant PROICO 30412, UNRC grant number 18/C417  and SCyT-FCEyN-UNLPam grant number PI67(M)}\\
Dpto. de Matem\'atica, Facultad de Ciencias Exactas y Naturales\\
Universidad Nacional de La Pampa\\
(6300) Santa Rosa, La Pampa, Argentina\\
sonia.acinas@gmail.com\\
\vspace{3mm}
Fernando D. Mazzone \thanks{The author is partially supported by a UNRC grant number 18/C417 and SCyT-FCEyN-UNLPam grant number PI67(M).}\\
\vspace{3mm}
Dpto. de Matem\'atica, Facultad de Ciencias Exactas, F\'{\i}sico-Qu\'{\i}micas y Naturales\\
Universidad Nacional de R\'{i}o Cuarto\\
(5800) R\'{\i}o Cuarto, C\'ordoba, Argentina,\\
fmazzone@exa.unrc.edu.ar\\}

\date{}

\begin{document}

\maketitle

\begin{abstract}

In this paper, we obtain existence results of periodic solutions of hamiltonian systems in the Orlicz-Sobolev space $\wphi([0,T])$.
We employ the direct method of calculus of variations and we consider  a potential  function $F$ satisfying the inequality
$|\nabla F(t,x)|\leq b_1(t) \Phi_0'(|x|)+b_2(t)$, with  $b_1, b_2\in L^1$ and  certain $N$-functions $\Phi_0$.
\end{abstract}

\section{Introduction}
This paper deals with system  of equations of the type:

\begin{equation}\label{ProbPrin-gral}
    \left\{%
\begin{array}{ll}
  \frac{d}{dt} D_{y}\mathcal{L}(t,u(t),u'(t))= D_{x}\mathcal{L}(t,u(t),u'(t)) \quad \hbox{a.e.}\ t \in (0,T),\\
    u(0)-u(T)=u'(0)-u'(T)=0,
\end{array}%
\right.
\end{equation}
where $\mathcal{L}:[0,T]\times\rr^d\times\rr^d\to\rr$, $d\geq 1$, is called the \emph{Lagrange function} or \emph{lagrangian} and the unknown function  $u:[0,T]\to\rr^d$ is absolutely continuous. In other words, we are interested in  finding \emph{periodic weak solutions} of \emph{Euler-Lagrange system of ordinary differential equations}.

The problem \eqref{ProbPrin-gral} comes from a variational one, that is,  the equation in  \eqref{ProbPrin-gral} is the Euler-Lagrange equation associated to the \emph{action integral}

\begin{equation}\label{integral_accion}
I(u)=\int_{0}^T \mathcal{L}(t,u(t),u'(t))\ dt.
\end{equation}

This topic was deeply addressed for the \emph{Lagrange function}
\begin{equation}\label{eq:lagrange_cuad}
\mathcal{L}_{p,F}(t,x,y)=\frac{|y|^p}{p}+F(t,x),
\end{equation}
for $1<p<\infty$. For example, the classic book  \cite{mawhin2010critical} deals mainly with problem \eqref{ProbPrin-gral} for the lagrangian $\mathcal{L}_{2,F}$ and through various methods: direct, dual action, minimax, etc. The results in \cite{mawhin2010critical} were extended and improved in several articles, see  \cite{tang1995periodic,tang1998periodic,wu1999periodic,tang2001periodic,zhao2004periodic}  to cite some examples. Lagrange functions \eqref{eq:lagrange_cuad} for arbitrary $1<p<\infty$ are considered in  \cite{Tian2007192,tang2010periodic} and in this case \eqref{ProbPrin-gral}  is reduced to the $p$-laplacian system
\begin{equation*}\label{ProbP-lapla}
    \left\{%
\begin{array}{ll}
   \frac{d}{dt}\left(u'(t)|u'|^{p-2}\right) = \nabla F(t,u(t)) \quad \hbox{a.e.}\ t \in (0,T),\\
    u(0)-u(T)=u'(0)-u'(T)=0.
\end{array}%
\right.
\end{equation*}

In this context, it  is customary to call $F$ a  \emph{potential function}, and it is assumed that $F(t,x)$ is differentiable with respect to $x$ for a.e. $t\in [0,T]$ and the following conditions hold:
\begin{enumerate}
\labitem{(C)}{item:condicion_c} $F$ and its gradient $\nabla F$, with respect to $x\in\rr^d$,  are  Carath\'eodory functions, i.e. they are measurable functions with respect to $t\in [0,T]$ for every  $x\in\rr^d$, and they are continuous functions with  respect to  $x\in\rr^d$ for a.e. $t \in [0,T]$.
 \labitem{(A)}{item:condicion_a}  For   a.e. $t\in [0,T]$,
\begin{equation}
|F(t,x)| + |\nabla F(t,x)|  \leq a(|x|)b(t).
\end{equation}
In this inequality, it is assumed that the function  $a:[0,+\infty)\to [0,+\infty)$ is continuous and non decreasing,
and $0\leq b\in L^1([0,T],\rr)$.
\end{enumerate}

In \cite{ABGMS2015} it was treated  the case of a lagrangian $\mathcal{L}$ which is lower bounded by a Lagrange function
\begin{equation}\label{eq:lagrange_phi}
\mathcal{L}_{\Phi,F}(t,x,y)=\Phi(|y|)+F(t,x),
\end{equation}
where  $\Phi$ is an $N$-function (see section \ref{preliminares} for the definition of this concept).
In the paper \cite{ABGMS2015} it was also assumed  a condition of \emph{bounded oscillation} on $F$.
In this current paper we will study a condition of \emph{sublinearity} (see  \cite{tang1998periodic,wu1999periodic,zhao2004periodic,tang2010periodic,zhao2005existence}) on $\nabla F$ for the lagrangian  $\mathcal{L}_{\Phi,F}$, or more generally for lagrangians which are lower bounded by $\mathcal{L}_{\Phi,F}$.

The paper is organized as follows. In section \ref{preliminares}, we give preliminaries facts on $N$-functions and Orlicz-Sobolev spaces of functions. Section 3 is devoted to the main result of this work and it also includes an auxiliary lemma of vital importance. Section \ref{sec:proofs} contains the proofs and section \ref{sec:examples} provides an application of our result to a concrete case.

\section{Preliminaries}\label{preliminares}

For reader convenience, we give a short introduction to Orlicz and Orlicz-Sobolev spaces of vector valued functions. Classic references for these topics are \cite{adams_sobolev,KR,rao1991theory,M}.

Hereafter we denote  by $\mathbb{R}^+$  the set of all non negative real numbers. A function $\Phi:\mathbb{R}^+\to \mathbb{R}^+ $ is called an \emph{$N$-function} if $\Phi$ is convex and it also satisfies that
\[
\lim_{t\to+\infty}\frac{\Phi(t)}{t}=+\infty\quad\text{and}\quad \lim_{t\to 0}\frac{\Phi(t)}{t}=0.
\]
In addition,  in this paper for the sake of simplicity  we assume that $\Phi$ is differentiable and we call $\varphi$  the derivative of $\Phi$.
On these assumptions, $\varphi:\mathbb{R}^+\rightarrow \mathbb{R}^+$ is a homeomorphism whose inverse will be denoted by $\psi$.
We write $\Psi$ for the primitive of $\psi$ that satisfies $\Psi(0)=0$. Then, $\Psi$ is an $N$-function which  is known as the \emph{complementary function} of $\Phi$.

 We recall that an $N$-function $\Phi(u)$ has \emph{principal part} $f(u)$ if $\Phi(u)=f(u)$ for large values of the argument (see \cite[p. 16]{KR} and \cite[Sec. 7]{KR} for  properties of principal part).

There exist several orders and equivalence relations between $N$-functions (see \cite[Sec. 2.2]{rao1991theory}).
Following \cite[Def. 1, pp. 15-16]{rao1991theory} we say that the   $N$-function $\Phi_2$ is \emph{stronger} than the $N$-function  $\Phi_1$, in symbols  $\Phi_1\prec\Phi_2$, if  there exist $a>0$ and $x_0\geq 0$ such that
\begin{equation}\label{eq:prec}\Phi_1(x)\leq \Phi_2(ax), \quad x\geq x_0.\end{equation}
 The $N$-functions  $\Phi_1$ and   $\Phi_2$ are \emph{equivalent} ($\Phi_1\sim\Phi_2$)  when  $\Phi_1\prec\Phi_2$ and $\Phi_2\prec\Phi_1$.
We say that  $\Phi_2$ is \emph{essentially stronger} than  $\Phi_1$  ($\Phi_1\llcurly\Phi_2$) if and only if for every $a>0$ there exists $x_0=x_0(a)\geq 0$ such that \eqref{eq:prec} holds. Finally, we say that  $\Phi_2$ is \emph{completely stronger} than  $\Phi_1$  ($\Phi_1\closedprec\Phi_2$) if and only if for every $a>0$ there exist $K=K(a)>0$ and  $x_0=x_0(a)\geq 0$ such that

\begin{equation}\label{eq:prec2}\Phi_1(x)\leq K\Phi_2(ax), \quad x\geq x_0.\end{equation}

We also say that a non decreasing function $\eta:\mathbb{R}^+\rightarrow \mathbb{R}^+$ satisfies the  \emph{$\Delta_2^{\infty}$-condition}, denoted by $\eta \in \Delta_2^{\infty}$,
if there exist  constants $K>0$ and  $x_0\geq 0$ such that
\begin{equation}\label{delta2defi}\eta(2x)\leq K\eta(x),
\end{equation}
for every $x\geq x_0$. We note that $\eta \in \Delta_2^{\infty}$ if and only if $\eta\closedprec\eta$.
If $x_0=0$,  the function   $\eta:\mathbb{R}^+\rightarrow \mathbb{R}^+$ is said to satisfy the
\emph{$\Delta_2$-condition} ($\eta \in \Delta_2$).
If there exists $x_0>0$ such that  inequality \eqref{delta2defi} holds for $x\leq x_0$,
we will say that $\Phi$ satisfies the
\emph{$\Delta_2^0$-condition} ($\Phi\in\Delta_2^0$).

We denote by $\alpha_{\eta}$ and $\beta_{\eta}$ the so called  \emph{Matuszewska-Orlicz indices} of the function $\eta$, which are defined next. Given
an increasing, unbounded, continuous function   $\eta:[ 0,+\infty)\to [0,+\infty)$ such that $\eta(0)=0$, we define
\begin{equation}\label{MO_indices}
    \alpha_{\eta}:=\lim\limits_{t\to 0^{+}}\frac{\log \left (\sup\limits_{u>0}\frac{\eta(t u)}{\eta(u)} \right ) }{\log(t)},\quad
    \beta_{\eta}:=\lim\limits_{t\to +\infty}\frac{\log \left  (\sup\limits_{u>0}\frac{\eta(t u)}{\eta(u)}\right )}{\log(t)}.
\end{equation}
It is known that the previous limits exist and  $0\leq \alpha_{\eta}\leq \beta_{\eta}\leq +\infty$
(see \cite[p. 84]{M}). The relation $\beta_{\eta}<+\infty$ holds true if and only if $\eta \in \Delta_2$
(\cite[Thm. 11.7]{M}). If $(\Phi,\Psi)$ is a complementary pair of  $N$-functions then
\begin{equation}\label{compl_ind}
 \frac{1}{\alpha_{\Phi}}+\frac{1}{\beta_{\Psi}}=1,
\end{equation}
(see \cite[Cor. 11.6]{M}). Therefore $1\leq \alpha_{\Phi}\leq\beta_{\Phi}\leq \infty $.

 If $\eta$ is an increasing function that satisfies the $\Delta_2$-condition, then $\eta$ is controlled by above and below
 by power functions (\cite[Sec. 1]{Gustavsson1977}, \cite[Eq. (2.3)-(2.4)]{fiorenza1997indices} and \cite[Thm. 11.13]{M}).   More concretely, for every $\epsilon>0$ there exists a
constant $K=K(\eta,\epsilon)$ such that, for every $t,u\geq 0$,
\begin{equation}\label{delta2-potencias}
    K^{-1}\min\big\{t^{\beta_{\eta}+\epsilon},t^{\alpha_{\eta}-\epsilon} \big\}\eta(u)\leq \eta(t u)\leq
    K\max\big\{t^{\beta_{\eta}+\epsilon},t^{\alpha_{\eta}-\epsilon} \big\}\eta(u).
\end{equation}

Let $d$ be a positive integer. We denote by $\mathcal{M}:=\mathcal{M}([0,T],\rr^d)$  the set of all measurable functions defined on $[0,T]$ with values on $\mathbb{R}^d$ and  we write $u=(u_1,\dots,u_d)$ for  $u\in \mathcal{M}$. For the set of functions $\mathcal{M}$, as for other similar sets, we will omit the reference to codomain $\mathbb{R}^d$ when $d=1$.

Given  an $N$-function $\Phi$ we define the \emph{modular function}
$\rho_{\Phi}:\mathcal{M}\to \mathbb{R}^+\cup\{+\infty\}$ by
\[\rho_{\Phi}(u):= \int_0^T \Phi(|u|)\ dt.\]
Here $|\cdot|$ is the euclidean norm of $\mathbb{R}^d$.
Now, we introduce the \emph{Orlicz class} $C^{\Phi}=C^{\Phi}([0,T],\rr^d)$   by setting
\begin{equation}\label{claseOrlicz}
  C^{\Phi}:=\left\{u\in \mathcal{M} | \rho_{\Phi}(u)< \infty \right\}.
\end{equation}
The \emph{Orlicz space} $\lphi=L^{\Phi}([0,T],\rr^d)$ is the linear hull of $\claseor$;
equivalently,
\begin{equation}\label{espacioOrlicz}
\lphi:=\left\{ u\in \mathcal{M}| \exists \lambda>0: \rho_{\Phi}(\lambda u) < \infty   \right\}.
\end{equation}
  The Orlicz space $\lphi$ equipped with the \emph{Orlicz norm}
\[
\|  u  \orlnor:=\sup \left\{  \int_0^T u\b{\cdot} v\ dt \big| \rho_{\Psi}(v)\leq 1\right\},
\]
is a Banach space. By $u\b{\cdot} v$ we denote the usual dot product in $\mathbb{R}^{d}$ between $u$ and $v$.

The following  inequality holds for any $u\in\lphi$
\begin{equation}\label{amemiya-ine}
\|u\orlnor\leq \frac{1}{k}\left\{1+\rho_{\Phi}(ku)\right\},\quad\text{for every } k>0.
\end{equation}
In fact, $\|u\orlnor$ is the infimum for $k>0$ of the right hand side in above expression  (see \cite[Thm. 10.5]{KR} and \cite{hudzik2000amemiya}).

The subspace $\ephi=\ephi([0,T],\rr^d)$ is defined as the closure in $\lphi$ of the subspace $L^{\infty}([0,T],\rr^d)$ of all $\mathbb{R}^d$-valued essentially bounded functions. It has shown that  $\ephi$ is the only one maximal subspace contained in the Orlicz class $\claseor$, i.e.
$u\in\ephi$ if and only if $\rho_{\Phi}(\lambda u)<\infty$ for any $\lambda>0$. The equality $\lphi=\ephi$ is true if and only if $\Phi\in\Delta_2^{\infty}$.

A generalized version of \emph{H\"older's inequality} holds in Orlicz spaces (see \cite[Thm. 9.3]{KR}). Namely, if $u\in\lphi$ and $v\in\lpsi$ then $u\cdot v\in L^1$ and
\begin{equation}\label{holder}
\int_0^Tv\cdot u\ dt\leq \|u\orlnor\|v\|_{L^{\Psi}}.
\end{equation}

If $X$ and $Y$ are  Banach spaces such that  $Y\subset X^*$, we denote by $\langle\cdot,\cdot\rangle:Y\times X\to\mathbb{R}$ the bilinear pairing  map given by $\langle x^*,x\rangle=x^*(x)$. H\"older's inequality shows that $\lpsi\subset \left[\lphi\right]^*$, where the pairing
$\langle v, u\rangle$
is defined by
\begin{equation}\label{pairing}
  \langle v,u\rangle=\int_0^Tv\cdot u\ dt,
\end{equation}
with  $u\in\lphi$ and $v\in\lpsi$.
 Unless $\Phi \in \Delta_2^{\infty}$, the relation $\lpsi= \left[\lphi\right]^*$ will not be satisfied.
In general, it is true  that  $\left[\ephi\right]^*=\lpsi$.

We define the \emph{Sobolev-Orlicz space} $\wphi$ (see \cite{adams_sobolev}) by
\[\wphi:=\{u| u \hbox{ is absolutely continuous on $[0,T]$ and } u'\in \lphi\}.\]
$\wphi$ is a Banach space when equipped with the norm
\begin{equation}\label{def-norma-orlicz-sob}
\|  u  \|_{\wphi}= \|  u  \|_{\lphi} + \|u'\orlnor.
\end{equation}
And, we introduce the following subset of $\wphi$
\begin{equation}\label{def-esp-orlicz-sob-per}
\begin{split}
\wphi_T&=\{u\in\wphi|u(0)=u(T)\}.
\end{split}
\end{equation}

We will use repeatedly the decomposition $u=\overline{u}+\widetilde{u}$ for a function $u\in L^1([0,T])$  where $\overline{u} =\frac1T\int_0^T u(t)\ dt$ and $\widetilde{u}=u-\overline{u}$.

As usual, if $(X,\|\cdot\|_X)$ is a Banach space and $(Y,\|\cdot \|_Y)$ is a subspace of $X$,  we write $Y\hookrightarrow X$ and we say that $Y$ is \emph{embedded} in $X$  when the restricted identity map $i_Y:Y\to X$ is bounded. That is, there exists $C>0$ such that
$\|y\|_X\leq C\|y\|_Y$ for any $y\in Y$.  With this notation, H\"older's inequality states that  $\lpsi\hookrightarrow  \left[\lphi\right]^*$; and, it is easy to see that for every $N$-function $\Phi$ we have that $L^{\infty}\hookrightarrow\lphi \hookrightarrow L^1$.

The following simple  embedding lemma, whose proof can be found in \cite{ABGMS2015}, will be used several times.

\begin{lem}\label{inclusion orlicz} For every $u\in\wphi$, $\widetilde{u}\in L^{\infty}$ and
\begin{align}
 \|u\|_{L^{\infty}} &\leq\Phi^{-1}\left(\frac{1}{T}\right)\max\{1,T\}\|u\sobnor&\text{  (Sobolev's inequality).}\label{sobolev}\\
 \|\widetilde{u}\|_{L^{\infty}} &\leq T\Phi^{-1}\left(\frac{1}{T}\right)\|u'\orlnor&
\text{  (Sobolev-Wirtinger's inequality).}\label{wirtinger}
\end{align}

\end{lem}

\section{Main result}

We begin with a lemma which establishes the coercivity of the modular function $\rho_{\Phi}(u)$ with respect to certain functions of the Orlicz norm $\Phi_0(\|u\orlnor|)$. This lemma generalizes \cite[Lemma 5.2]{ABGMS2015} in two directions.
Namely, certain power function is replaced by a more general $N$-function $\Phi_0$ and
the $\Delta_2$-condition  on $\Psi$ is relaxed to $\Delta_2^{\infty}$.
It is worth noting that the second improvement is more important than the first one.
And, we present the result here since the lemma introduces a function $\Phi^*$ that will play a significant role in the statement of our main theorem.

\begin{lem}\label{lem_coer}
Let $\Phi,\Psi$ be complementary $N$-functions with $\Psi \in \Delta_2^{\infty}$. Then,
there exists an $N$-function $\Phi^*$ being $\Phi^*\prec\Phi$,
such that  for every $N$-function $\Phi_0$ that satisfies $\Phi_0\llcurly\Phi^*$ and for every $k>0$, we have
\begin{equation}\label{eq:coer_mod}
\lim\limits_{\|u\orlnor\to \infty}
\frac{\int_0^T \Phi(|u|)\,dt}{\Phi_0(k\|u\orlnor)}=\infty.
\end{equation}
Reciprocally, if  \eqref{eq:coer_mod} holds for some $N$-function $\Phi_0$,  then $\Psi\in\Delta_2^{\infty}$.
\end{lem}

 We point out that this lemma can be applied to more cases than \cite[Lemma 5.2]{ABGMS2015}. For example, if $\Phi(u)=u^2$, $\Phi_1$ and $\Phi_0$ are  $N$-functions with principal parts equal to $u^2/\log u$ and $u^2/(\log u)^2$ respectively, then  \eqref{eq:coer_mod} holds for $\Phi_0$. On the other hand, $\Phi_0(|u|)$ is not dominated for any  power function $|u|^{\alpha}$ with $\alpha<2$.

As in  \cite{ABGMS2015} we will consider general Lagrange functions
$\mathcal{L}:[0,T]\times\rr^d\times\rr^d\to\rr$
satisfying the structure conditions
\begin{align}
|\mathcal{L}(t,x,y)| &\leq a(|x|)\left(b(t)+ \Phi\left(\frac{|y|}{\lambda}+f(t) \right)\right),&\tag{$A_1$}\label{eq:estru1}
\\
|D_{x}\mathcal{L}(t,x,y)| &\leq a(|x|)\left(b(t)+ \Phi\left(\frac{|y|}{\lambda}+f(t) \right)\right),&\tag{$A_2$}\label{eq:estru2}
\\
|D_{y}\mathcal{L}(t,x,y)| &\leq a(|x|)\left(c(t)+ \varphi\left(\frac{|y|}{\lambda}+f(t)\right)\right),
&\tag{$A_3$}\label{eq:estru3}
\end{align}
where $a\in C(\mathbb{R}^+,\mathbb{R}^+)$, $\lambda>0$, $\Phi$ is an $N$-function,
$\varphi$ is the continuous derivative of $\Phi$,
$b\in L^1_1([0,T])$,  $c\in\lpsi_1([0,T])$ and  $f\in \ephi_1([0,T])$. We denote by $\mathfrak{A}(a,b,c,\lambda,f,\Phi)$ the set of all Lagrange functions satisfying  \eqref{eq:estru1}, \eqref{eq:estru2} and \eqref{eq:estru3}.

In \cite{ABGMS2015} it was shown that if $\mathcal{L}\in \mathfrak{A}(a,b,c,\lambda,f,\Phi)$
then there  exists the Gate\^aux derivative of the integral functional $I$ defined by \eqref{integral_accion},
on the set
\[\domi:=\{u\in\wphi([0,T],\rr^d): d(u',\ephi)<\lambda\}.\]
We observe that the condition \ref{item:condicion_a} on the potential $F$ is equivalent to say that $\mathcal{L}_{\Phi,F}\in \mathfrak{A}(a,b,0,1,0,\Phi)$.

 Unlike what is usual in the literature, we do not assume the lagrangian $\mathcal{L}$  split into two terms,
one of them function of $y$ and the other one function of $(t,x)$.
We only suppose that $\mathcal{L}$ is lower bounded by a function of this type. More precisely, we assume that for every $(t,x,y)\in\rr\times\rr^d\times\rr^d$
\begin{equation}\label{eq:control_inf}
\mathcal{L} \geq \mathcal{L}_{\Phi,F},\quad \text{with}\;F\text{ satisfying \ref{item:condicion_a} \;and\; \ref{item:condicion_c},
\;\text{and}\;$\Phi$ being an $N$-function}.\tag{$A_4$}
\end{equation}
In addition, as usual we suppose that the time integral of $F$ satisfies certain coercivity condition, see \eqref{eq:propiedad-coercividad-phi0} below.  However,  all these hypotheses are not enough. It is also necessary to assume extra conditions on the potential $F$.
Several hypotheses were tested in the past years. The so called \emph{subconvexity} of $F$ was tried in \cite{wu1999periodic,tang1995periodic,zhao2004periodic} for semilinear equations and in \cite{xu2007some,tang2010periodic} for $p$-laplacian systems. Potentials $F$ satisfying the following inequality
\begin{equation}\label{holder_cont}
  \left| F(t,x_2)- F(t,x_1) \right|\leq b_1(t)(1+|x_2-x_1|^{\mu}).
\end{equation}
 were studied in \cite{ABGMS2015}.
Regarding \eqref{holder_cont}, it is interesting to notice that such inequality is equivalent to say the condition
$\|F(t,\cdot)\|_{BO}\in L^1([0,T])$, where $\|\cdot\|_{BO}$ denotes the seminorm defined in  \cite[p. 125]{zhu2012analysis} on the space of functions of bounded variations.

 In \cite{tang1998periodic, tang2010periodic} the authors  dealt with the $p$-laplacian case
with potentials $F$ such that
\begin{equation}\label{eq:cota_pot} |\nabla F(t,x)|\leq b_1(t)|x|^{\alpha}+b_2(t),
 \end{equation}
where  $b_1,b_2 \in L^1([0,T])$ and $\alpha<p$. Such potentials $F$
were called  \emph{sublinear nonlinearities}. In this paper, we are interested in studying this type of potentials,
but with more general bounds on $\nabla F$ which include $N$-functions instead of power functions;
namely, we will consider inequalities like
\begin{equation}\label{holder_cont-mu}
  \left| \nabla F(t,x) \right|\leq b_1(t)\Phi_0'(|x|)+b_2(t),
  \tag{$A_5$}
\end{equation}
with $\Phi_0$  a differentiable $N$-function and $b_1,b_2 \in L^1([0,T])$.

Next, we give our main result.  Here,  we will amend an
erroneous assumption made in the end of the proof of  \cite[Thm. 6.2]{ABGMS2015}.
There, it was assumed without discussion that  a minimum of  $I$ was on the domain of differentiability of $I$.

\begin{thm}\label{coercitividad-r}
Let $\Phi$ be an $N$-function whose complementary function $\Psi$ satisfies the $\Delta_2^{\infty}$-condition
and suppose that $\Phi^*$ is given by Lemma \ref{lem_coer}. Assume that the lagrangian $\mathcal{L}\in \mathfrak{A}(a,b,c,\lambda,f,\Phi)$ satisfies \eqref{eq:control_inf}, where
the potential $F$ fulfills  \ref{item:condicion_c}, \ref{item:condicion_a} and
the following conditions:
\begin{enumerate}
 \item  \eqref{holder_cont-mu}  for some $N$-function  $\Phi_0$ such that $\Phi_0\llcurly\Phi^*$.
 \item \begin{equation}\label{eq:propiedad-coercividad-phi0}
\lim_{|x|\to\infty}\frac{\int_{0}^{T}F(t,x)\ dt}{\Psi_1(\Phi_0'(2|x|))}=+\infty,\tag{$A_6$}
\end{equation}
for some $N$-function $\Psi_1$ with complementary function $\Phi_1$ satisfying $\Phi_0\llcurly \Phi_1\llcurly \Phi^*$.
\end{enumerate}

Then the action integral $I$ has a minimum $u\in\wphi([0,T],\rr^d)$ such that $d(u',\ephi)\leq \lambda$. If $d(u',\ephi)< \lambda$, the lagrangian $\mathcal{L}(t,x,y)$ is strictly convex  with respect to $y\in\rr^d$ and
$D_y\mathcal{L}(0,x,y)=D_y\mathcal{L}(T,x,y)$
then $u$ solves the problem \eqref{ProbPrin-gral}.

\end{thm}

\begin{comentario} If $\Phi\in\Delta_2^{\infty}$ the the condition  $d(u',\ephi)\leq \lambda$ is automatically satisfied.
\end{comentario}

\section{Proofs}\label{sec:proofs}

The following result is analogous to some lemmata in $W^{1,p}$, see \cite[Lem. 1]{xu2007some}.
\begin{lem}\label{infinito-a-prom-upunto}
 $\|u\sobnor\to \infty$ iff  $(|\overline{u}|+\|u'\orlnor)\to \infty$.
\end{lem}

\begin{proof}
By the decomposition $u=\overline{u}+\tilde{u}$ and some elementary operations,
we get
\begin{equation}\label{cota-u-lphi}
\|u\orlnor=
\|\overline{u}+\tilde{u}\orlnor\leq
\|\overline{u}\orlnor+\|\tilde{u}\orlnor=
|\overline{u}|\|1\orlnor+\|\tilde{u}\orlnor.
\end{equation}
It is known that $L^{\infty}\hookrightarrow\lphi$, i.e.
there exists $C_1=C_1(T)>0$ such that for any $\tilde{u}\in L^{\infty}$ we have
\[
\|\tilde{u}\orlnor
\leq
C_1 \|\tilde{u}\|_{L^{\infty}};
\]
and, applying  Sobolev's inequality,  we obtain Wirtinger's inequality,
that is there exists $C_2=C_2(T)>0$ such that
\begin{equation}\label{cota-u-tilde}
\|\tilde{u}\orlnor
\leq
C_2\|u'\orlnor.
\end{equation}
Therefore, from \eqref{cota-u-lphi}, \eqref{cota-u-tilde} and \eqref{def-norma-orlicz-sob},
we get
\[
\|u\sobnor\leq
C_3(|\overline{u}|+\|u'\orlnor)
\]
where $C_3=C_3(T)$. Finally, as $\|u\sobnor\to \infty$ we conclude that
$(|\overline{u}|+\|u'\orlnor)\to \infty$.

For the converse, we observe that
\[|\overline{u}|\leq \frac{1}{T}\|1\|_{\lpsi}\|u\orlnor.\]
Hence
\[|\overline{u}| +\|u'\orlnor\leq C_4(\|u\orlnor +\|u'\orlnor),\]
and the property under consideration is proved.
\end{proof}

\begin{lem}\label{lema:delta_2 y der} Let $\Phi$ be a not necessarily differentiable $N$-function and
let $\varphi$ be the right continuous derivative of $\Phi$. Then $\Phi\in\Delta_2^{\infty}$ (
$\Phi\in\Delta_2$) iff $\varphi\in\Delta_2^{\infty}$ (
$\varphi\in\Delta_2$).
\end{lem}
\begin{proof}  It is consequence of \cite[Thm. 11.7]{M} and \cite[Rem. 5, p. 87]{M}.

\end{proof}

The following lemma improves the result on the comment at the beginning of \cite[p. 24]{KR}.

\begin{lem}\label{lem:delta2-equiv-delta2-global} Let $\Psi$ be an $N$-function satisfying the $\Delta_2^{\infty}$-condition. Then there exists an $N$-function $\Psi^*$  such that $\Psi^*\in\Delta_2$, $\Psi\leq\Psi^*$ and for every $a>1$ there exists $x_0=x_0(a)\geq 0$ such that $\Psi^*(x)\leq a\Psi(x)$, for every $x\geq x_0$.  In particular,  every $\Delta_2$ near infinity  $N$-function is equivalent to a $\Delta_2$  $N$-function.
\end{lem}

\begin{proof} We can assume that $\Psi \notin \Delta_2^0$.
Consequently, from Lemma \ref{lema:delta_2 y der} we have that the right continuous derivative $\psi$ of $\Psi$ does not satisfy the $\Delta_2^0$-condition.
Therefore, we obtain a sequence of positive numbers $x_n$, $n=1,2,\ldots$,  such that $x_n\to 0$,
\begin{equation}\label{eq:prop_psi*}2x_{n+1}<x_n<2x_n\quad\text{and}\quad \psi(2x_n)> 2\psi(x_n).
 \end{equation}

We define $\psi^*$ inductively for $n$ on the interval $[2x_n,+\infty)$ of the following way.
We put $\psi^*(x)=\psi(x)$ when  $x\in[ 2x_1,+\infty)$.
Suppose $\psi^*$ defined on $[2x_n,+\infty)$ and we set $\psi^*$ on $[2x_{n+1},2x_n)$ by
\[
 \psi^*(x)=\left\{
\begin{array}{cc}
  \max\left\{\psi(x),\frac{\psi^*(2x_n)}{2x_n}(x-x_n)+\frac{ \psi^*(2x_n)}{2}\right\},& \text{if } x_n\leq x<2x_n\\
   \frac{\psi^*(2x_n)}{2}& \text{if }2x_{n+1} \leq x<x_n
  \end{array}
\right.
\]
Moreover, we define $\psi^*(0)=0$.

Next, we will use induction  to prove that
\begin{enumerate}
 \item\label{item:prop_psi*1} $\psi^*(x_n)=\frac12\psi^*(2x_n)$,
 \item\label{item:prop_psi*2}  $\psi^*$ is non-decreasing $[2x_{n},+\infty)$,
 \item\label{item:prop_psi*3}   $\psi\leq \psi^*$ in  $[2x_{n},+\infty)$.
\end{enumerate}

We suppose $n=1$. Then items \ref{item:prop_psi*2} and \ref{item:prop_psi*3} are obvious. From \eqref{eq:prop_psi*} we have
\[\psi(x_1)<\frac12\psi(2x_1)=\frac12\psi^*(2x_1),\]
and this inequality implies \ref{item:prop_psi*1}.

Assume that properties \ref{item:prop_psi*1}-\ref{item:prop_psi*3} hold for $n>1$.  Clearly $\psi^*$ is non decreasing on each interval $[2x_{n+1},x_n)$ and $[x_n,2x_n)$. Since $\psi$ is right continuous and $\psi(x_n)<2^{-1}\psi(2x_n)\leq 2^{-1}\psi^*(2x_n)$, then $\psi^*$ is continuous at $x_n$. Therefore, $\psi^*$ is non decreasing on $[2x_{n+1},2x_n)$. Suppose $x\in[2x_{n+1},2x_n)$ and $y\geq 2x_n$.  From the definition of $\psi^*$, the inductive hypothesis, item \ref{item:prop_psi*3} and item \ref{item:prop_psi*2}, we obtain
\[\psi^*(x)\leq\max\{\psi(2x_n),\psi^*(2x_n)\}=\psi^*(2x_n)\leq\psi^*(y).\]
This proves item \ref{item:prop_psi*2} on the interval $[2x_{n+1},+\infty)$.
Inequality in item \ref{item:prop_psi*3} holds by inductive hypothesis on $[2x_n,+\infty)$
and it is obvious for  $x\in[x_n,2x_n)$.
If $ x\in[2x_{n+1},x_n)$, then $\psi(x)\leq\psi(x_n)\leq\psi^*(x_n)=\psi^*(x)$.
This proves  \ref{item:prop_psi*3} on the interval $[2x_{n+1},+\infty)$.

Now, using \eqref{eq:prop_psi*} and the already proved item \ref{item:prop_psi*3} for $n+1$, we deduce $\psi(x_{n+1})<2^{-1}\psi(2x_{n+1})\leq 2^{-1}\psi^*(2x_{n+1})$.
Then
\[\psi^*(x_{n+1})=\max\left\{\psi(x_{n+1}),\frac12\psi^*(2x_{n+1}) \right\}=\frac12\psi^*(2x_{n+1}),\]
i.e. we have just proved item \ref{item:prop_psi*1}.

We note that $\psi^*(x_{n+1})=2^{-1}\psi^*(2x_{n+1})\leq 2^{-1}\psi^*(x_{n})$. Consequently $\psi^*(x)\to 0$ when $x\to 0$.
Therefore $\psi^*$ is right continuous at $0$ and indeed it is right continuous on $[0,+\infty)$.
Moreover, since $\psi(x)=\psi^*(x)$ for $x\geq 2x_1$ being $\psi$ the right continuous derivative of an $N$-function, $\psi^*(x)\to +\infty$ when $x\to +\infty$. In this way,
\[\Psi^*(x):=\int_0^x\psi^*(t)dt\]
defines an $N$-function.

Let's see that $\psi^*$ satisfies the $\Delta_2$-condition.
It is sufficient to prove that $\psi^*$ satisfies the $ \Delta_2^0$-condition.
To this end, suppose that  $x\leq x_1$ and take $n\in\mathbb{N}$ such that $x_{n+1}\leq x\leq x_n$. Then
\[\psi^*(2x)\leq \psi^*(2x_n)=2\psi^*(2x_{n+1})=4\psi^*(x_{n+1})\leq 4\psi^*(x).\]
Thus,  $\Psi^* \in  \Delta_2$ and $\Psi\leq \Psi^*$.

It remains to show the inequality $\Psi^*(x)\leq a\Psi(x)$,
for every $a>1$ and sufficiently large $x$. We take $x_0$ sufficiently large to have
\[\frac{1}{a-1}\int_0^{2x_1}\psi^*(t)-\psi(t)dt<\Psi(x_0).\]
Therefore, if  $x>\max\{x_0,2x_1\}$ then
\[\Psi^*(x)=\Psi(x)+\int_0^{2x_1}\psi^*(t)-\psi(t)dt<\Psi(x)+(a-1)\Psi(x)= a\Psi(x).\]

The last statement of the lemma is consequence of   $\Psi(ax)>a\Psi(x)$ when $a>1$.

\end{proof}

The following lemma is essentially known,  because it is basically a consequence  of the fact that $\Psi\in\Delta_2^{\infty}$ if and only if $\Psi\closedprec\Psi$,  \cite[Prop. 4, p. 20]{rao1991theory} and \cite[Cor. 10, p. 30]{rao1991theory}.
However, we prefer to include an alternative proof, as we do not see clearly that the results of \cite{rao1991theory} contemplate the case of  $N$-functions satisfying the $\Delta_2$-condition.

\begin{lem}\label{lem:submultipliativa}
Let $\Phi,\Psi$ be complementary functions.
The next statements are equivalent:
\begin{enumerate}
\item\label{item1} $\Psi \in \Delta_2$ ($\Psi \in \Delta_2^{\infty}$).
\item\label{item2} There exists an $N$-function $\Phi^*$ such that
\begin{equation}\label{eq:caract_delta2}
\Phi(rs)\geq \Phi^*(r)\Phi(s)\quad\mbox{for every }r\geq1,s\geq 0\, (r\geq1,s\geq 1).
\end{equation}
\end{enumerate}
\end{lem}

\begin{comentario} We want to emphasize that the difference between the conclusions in item \ref{item2} of the previous lemma
is  that \eqref{eq:caract_delta2} holds for $s\geq 0$ or $s\geq 1$ depending on
$\Psi\in\Delta_2$ or $\Psi\in\Delta_2^{\infty}$,  respectively.

\end{comentario}

\begin{proof}
 In virtue of the comment that precedes the lemma, we only consider the case   $\Psi \in \Delta_2$.

\ref{item1})$\Rightarrow$\ref{item2}).
As a consequence of the $\Delta_2$-condition on $\Psi$, \eqref{compl_ind} and \eqref{delta2-potencias},
we get for every $1<\nu<\alpha_{\Phi}$ a constant $K=K_{\nu}>0$  such that

\begin{equation}\label{delta2-consecuencia}
\Phi(r s)\geq Kr^{\nu}\Phi(s),
\end{equation}
for any $1<\nu<\alpha_{\Phi}$,  $s\geq 0$ and $r>1$. This proves  \eqref{eq:caract_delta2} with $\Phi^*(r)=kr^\nu$, which is an $N$-function.

\ref{item2})$\Rightarrow$\ref{item1})
Next, we follow  \cite[p. 32, Prop. 13]{rao1991theory} and \cite[p. 29, Prop. 9]{rao1991theory}.
Assume that
\[
\Phi^*(r)\Phi(s)\leq \Phi(rs)\;\;r>1,\;s\geq 0.
\]
Let $u=\Phi^*(r)\geq \Phi^*(1)$ and $v=\Phi(s)\geq 0$. By a well known inequality \cite[p. 13, Prop. 1]{rao1991theory} and \eqref{eq:caract_delta2},   for $u\geq \Phi^*(1)$ and $v> 0$ we have
\[
\frac{uv}{\Psi^{-1}(uv)}\leq \Phi^{-1}(uv)\leq{\Phi^*}^{-1}(u)\Phi^{-1}(v)\leq
\frac{4uv}{{\Psi^*}^{-1}(u)\Psi^{-1}(v)},
\]
then
\[
{\Psi^*}^{-1}(u)\Psi^{-1}(v)\leq 4 \Psi^{-1}(uv).
\]
If we take $x=\Psi^{-1}_1(u)\geq \Psi^{-1}_1(\Phi^*(1))$ and $y=\Psi^{-1}(v)\geq 0$, then
\[
\Psi\left(\frac{xy}{4}\right)\leq \Psi^*(x)\Psi(y).
\]
Now, taking  $x\geq \max\{8,{\Psi^*}^{-1}(\Phi^*(1))\}$ we get that $\Psi \in \Delta_2$.
\end{proof}

\begin{comentario} Note that if $\Phi^*$ satisfies \eqref{eq:caract_delta2} then $\Phi^*\prec \Phi$.
 \end{comentario}

\noindent\textbf{Proof Lemma \ref{lem_coer}} First, we suppose that $\Psi \in \Delta_2$.
Let $\Phi^*$ be an $N$-function satisfying \eqref{eq:caract_delta2}.
By the inequality \eqref{amemiya-ine}, for $r>1$ we have
\[
\int_0^T \Phi(|u|)\,dt\geq
\Phi^*(r) \int_0^T \Phi(r^{-1}|u|)\,dt\geq
\Phi^*(r)\{r^{-1}\|u\orlnor-1\}.
\]
Now, we choose $r=\frac{\|u\orlnor}{2}$; and, as $\|u\orlnor\to\infty$ we can assume $r>1$.
From \cite[Thm. 2 (b)(v), p. 16]{rao1991theory} and $\Phi_0\llcurly \Phi^*$,  we get
\begin{equation}\label{eq:caso-delta-2}
\lim\limits_{\|u\orlnor \to \infty} \frac{\int_0^T \Phi(|u|)\,dt}{\Phi_0(k\|u\orlnor)}\geq
\lim\limits_{\|u\orlnor \to \infty} \frac{\Phi^*\left(\frac{\|u\orlnor}{2}\right)}{\Phi_0(k\|u\orlnor)}
=\infty.
\end{equation}
Now, if $\Psi\in\Delta_2^{\infty}$ but $\Psi\notin\Delta_2$, we  use Lemma \ref{lem:delta2-equiv-delta2-global}.
Then, there exists an $N$-function $\Psi_1$ such that $\Psi_1\in\Delta_2$ and  $\Psi_1\sim\Psi\leq \Psi_1$.
Let $\Phi_1$ be the complementary function of $\Psi_1$. Then $\Phi\sim\Phi_1\leq \Phi$ (see \cite[Thm. 3.1]{KR}) and $\|\cdot\orlnor$ and $\|\cdot\|_{L^{\Phi_1}}$ are equivalent norms  (see \cite[Thm. 13.2 and Thm. 13.3]{KR}).
Thus, there exists an $N$-function $\Phi_1^*\prec \Phi_1$ (consequently $\Phi_1^*\prec \Phi$)
satisfying  \eqref{eq:coer_mod}
with the $N$-functions $\Phi_1$ and $\Phi_1^*$ instead of $\Phi$ and $\Phi^*$, respectively.
Let $C>0$ be a constant such that  $\|\cdot\orlnor\leq C\|\cdot\|_{L^{\Phi_1}}$. Then

\[\lim\limits_{\|u\orlnor \to \infty}\frac{\int_0^T \Phi(|u|)\,dt}{\Phi_0(k\|u\orlnor)}\geq \lim\limits_{\|u\orlnor \to \infty} \frac{\int_0^T \Phi_1(|u|)dt}{\Phi_0(kC\|u\|_{L^{\Phi_1}})}=+\infty.\]

Finally, if $\Phi_0$ is an $N$-function, then $\Phi_0(x)\geq \alpha |x|$ for  $\alpha$ small enough and $|x|>1$.
Therefore \eqref{eq:coer_mod} holds for $\Phi_0(x)=|x|$, then \cite[Lemma 5.2]{ABGMS2015}
implies  $\Psi\in\Delta_2^{\infty}$. \qed

\begin{defi}We define the  functionals $J_{C,\varphi}:\lphi\to (-\infty,+\infty]$ and $  H_{C,\varphi}:\rr^n\to \rr$, where $C>0$ and $\varphi:[0,+\infty)\to [0,+\infty)$, by
\begin{equation}\label{func_phi}
  J_{C,\varphi}(u):= \rho_{\Phi}\left(u\right)-C\varphi\left(\|u\orlnor\right),
\end{equation}
 and

\begin{equation}\label{eq:functional_H-bis}
 H_{C,\varphi}(x):=\int_0^TF(t,x)dt-C\varphi(2|x|),
\end{equation}
respectively.
\end{defi}

\vspace{.4cm}

\noindent\textbf{Proof. Theorem \ref{coercitividad-r} }. By the decomposition $u=\overline{u}+\tilde{u}$,   Cauchy-Schwarz's inequality
and \eqref{holder_cont-mu}, we have
\begin{equation}\label{cota-diferencia-F}
\begin{split}
&\left|\int_0^T F(t,u)-F(t,\overline{u})\,dt\right|=
\left|\int_0^T \int_0^1 \nabla F(t,\overline{u}+s\tilde{u}(t))\ccdot \tilde{u}(t) \,ds \,dt\right|
\\
&\leq \int_0^T \int_0^1 b_1(t)\Phi_0'(|\overline{u}+s\tilde{u}(t)|)|\tilde{u}(t)|\,ds\,dt+
\int_0^T \int_0^1 b_2(t)|\tilde{u}(t)|\,ds\,dt
\\
&=:I_1+I_2.
\end{split}
\end{equation}
First, by H\"older's and Sobolev-Wirtinger's inequalities we estimate $I_2$ as follows
\begin{equation}\label{cota-i2}
I_2\leq \|b_2\|_{L^1} \|\tilde{u}\|_{L^{\infty}}\leq
C_1\|u'\orlnor,
\end{equation}
 where $C_1=C_1(\|b_2\|_{L^1}, T)$.

Note that, since $\Phi_0'$ is an increasing function  and $\Phi_0'(x)\geq 0$ for $x\geq 0$, then
$\Phi'_0(a+b)\leq \Phi'_0(2a)+\Phi'_0(2b)$ for every $a,b\geq 0$.
In this way, we have
\begin{equation}\label{pot-suma}
\Phi'_0(|\overline{u}+s\tilde{u}(t)|)\leq
\Phi'_0(2|\overline{u}|)+\Phi'_0(2\|\tilde{u}\|_{L^{\infty}}),
\end{equation}
for every $s \in [0,1]$.  Now,  inequality \eqref{pot-suma}, H\"older's and Sobolev-Wirtinger's inequalities imply that
\begin{equation}\label{cota-i1}
\begin{split}
I_1&
\leq \Phi'_0(2|\overline{u}|) \|b_1\|_{L^1} \|\tilde{u}\|_{L^{\infty}}+\Phi'_0(2\|\tilde{u}\|_{L^\infty})
 \|b_1\|_{L^1}\|\tilde{u}\|_{L^\infty}
\\
&\leq C_2 \bigg\{ \Phi'_0(2|\overline{u}|) \|u'\orlnor
+\Phi'_0(C_3\|u'\orlnor) \|u'\orlnor\bigg\},
\end{split}
\end{equation}
where $C_2=C_2(T, \|b_1\|_{L^1} )$ and $C_3=C_3(T)$.
Next, by Young's inequality with complementary functions $\Phi_1$ and $\Psi_1$, we obtain
\begin{equation}\label{cota-i1-parcial}
 \begin{split}
\Phi_0'(2|\overline{u}|) \|u'\orlnor
&\leq
\Psi_1(\Phi_0'(2|\overline{u}|))+
\Phi_1(\|u'\orlnor).
\end{split}
\end{equation}
We have that any $N$-function $\Phi_0$ satisfies the inequality $x\Phi_0'(x)\leq \Phi_0(2x)$ (see \cite[p. 17]{rao1991theory}).
Moreover, since $\Phi_0\llcurly\Phi_1$ there exists $x_0=x_0(\Phi_0,\Phi_1,T)\geq 0$ such that $\Phi_0(2C_3x)\leq \Phi_1(x)$
for every $x\geq x_0$. Therefore, $\Phi_0(2C_3x)\leq \Phi_1(x)+C_4$ with $C_4=\Phi_0(2x_0)$. The previous observations imply that
\begin{equation}\label{cota-i1-parcial-segunda}
\Phi_0'(C_3\|u'\orlnor) \|u'\orlnor
\leq
C_3^{-1}(\Phi_1(\|u'\orlnor)+C_4).
\end{equation}
From \eqref{cota-i1}, \eqref{cota-i1-parcial}, \eqref{cota-i1-parcial-segunda} and \eqref{cota-i2}, we have
\begin{equation}\label{cota-i1-i2}
\begin{split}
I_1+I_2
&
\leq C_5
\bigg\{
\Psi_1(\Phi_0'(2|\overline{u}|))
+\Phi_1(\|u'\orlnor)
+\|u'\orlnor +1
\bigg\},\\
\end{split}
\end{equation}
with $C_5$ depending on $\Phi_0, \Phi_1, T, \|b_1\|_{L^1}$ and $\|b_2\|_{L^1} $.

Finally, using  \eqref{eq:control_inf},  \eqref{cota-diferencia-F} and
\eqref{cota-i1-i2}, we get
\begin{equation}\label{cota_inf_I}
\begin{split}
I(u)&
\geq\rho_{\Phi}(u')+\int_0^TF(t,u)\ dt
\\
&=\rho_{\Phi}( u')+ \int_0^T \left[F(t,u)-F(t,\overline{u})\right]\ dt
+  \int_0^TF(t,\overline{u})\ dt
\\
&\geq \rho_{\Phi}( u')
-C_5 \Phi_1(\|u'\orlnor)
+\int_0^TF(t,\overline{u})\ dt-
C_5 \Psi_1(\Phi_0'(2|\overline{u}|))-
C_5
\\
&\geq
J_{C_5,\Phi_1}(u')
+H_{C_5, \Psi_1\circ\Phi_0'}(\overline{u})
-C_5.
\end{split}
\end{equation}

Let $u_n$ be  a sequence in $\wphi$ with
$\|u_n\sobnor\to\infty$ and we have to prove that $I(u_n)\to\infty$.
On the contrary, suppose  that for a subsequence,
still denoted by $u_n$, $I(u_n)$ is upper bounded, i.e. there exists $M>0$ such that $|I(u_{n})|\leq M$.
As $\|u_n\sobnor\to\infty$, from Lemma \ref{infinito-a-prom-upunto},  we have $|\overline{u}_n|+\|u'_n\orlnor\to \infty$. Passing to a subsequence is necessary, still denoted $u_n$, we can assume that $|\overline{u}_n|\to \infty$ or $\|u'_n\orlnor\to \infty$.
Now, Lemma \ref{lem_coer} implies that the functional $J_{C_5,\Phi_1}(u')$ is coercive;
and, by \eqref{eq:propiedad-coercividad-phi0},
the functional $H_{C_5,\Psi_1\circ\Phi_0'}(\overline{u})$ is also coercive, then
$J_{C_5,\Phi_1}(u'_n) \to \infty$ or $H_{C_5,\Psi_1\circ\Phi_0'}(\overline{u}_n)\to \infty$.
From the condition \ref{item:condicion_a}  on $F$, we have that on a bounded set the functional $H_{C_5,\Psi_1\circ\Phi_0'}(\overline{u}_n)$ is lower bounded and
also $J_{C_5,\Phi_0'}(u'_n)\geq 0$.
Therefore,  $I(u_n)\to\infty$ as $\|u_n\sobnor\to\infty$ which contradicts the initial assumption on the behavior of $I(u_n)$.

Let $\{u_n\}\subset \wphi$  be a  minimizing sequence for the problem
\linebreak $\inf\{I(u)|u\in\wphi\}$.
Since  $I(u_n)$, $n=1,2,\ldots$,  is upper bounded, the previous part of the proof shows that $\{u_n\}$ is norm bounded. Hence, by virtue of  \cite[Cor. 2.2]{ABGMS2015}, we can assume, taking a subsequence if necessary, that $u_n$ converges uniformly to a $T$-periodic continuous   function $u$. As $\{u'_n\}$ is a norm bounded sequence in $\lphi$,
there exists a subsequence, again denoted by $u'_n$, that converges to a function $v\in\lphi$ in the weak* topology of $\lphi$.
From this fact and the uniform convergence of $u_n$ to $u$, we obtain that
\[
\int_0^T\xi'\cdot u\ dt=\lim_{n\to\infty}\int_0^T\xi'\cdot u_n \ dt=
-\lim_{n\to\infty}\int_0^T\xi\cdot u'_n\ dt=-\int_0^T\xi\cdot v\ dt,
\]
for every $T$-periodic function $\xi\in C^{\infty}([0,T],\rr^d)\subset\epsi$.
Thus $v=u'$ a.e. $t\in [0,T]$ (see \cite[p. 6]{mawhin2010critical}) and $u\in\wphi([0,T],\rr^d)$.

Now, taking into account the relations $\left[L^1\right]^*=L^{\infty}\subset  \epsi$ and $\lphi\subset L^1$, we have that $u'_n$ converges to $u'$ in the weak topology of $L^1$. Consequently,  from the semicontinuity of $I$ (see \cite[Lemma 6.1]{ABGMS2015})  we get
\[I(u)\leq  \liminf_{n\to\infty}I(u_n)=\inf\limits_{v\in\wphi_T}I(v).\]
Hence $u\in \wphiet$ is a minimun of $I$ on $\wphi_T$.

For the second part of the theorem, assume that $u$ is a minimum of $I$ with $d(u',\ephi)<\lambda$.
Since $I$ is G\^ateaux differentiable at $u$ (see  \cite[Thm. 3.2]{ABGMS2015}),
therefore $I'(u)\in (\wphi_T)^{\perp}$. Thus,
\[\int_0^T D_y\mathcal{L}(t,u(t),u'(t))\cdot v'(t)dt =-\int_0^T D_x\mathcal{L}(t,u(t),u'(t))\cdot v(t)dt,\]
for every  $v\in \wphi_T$.
From \cite[Eq. (26)]{ABGMS2015} we have
\[D_y\mathcal{L}(t,u(t),u'(t))\in \lpsi([0,T],\rr^n)\hookrightarrow L^1([0,T],\rr^n);\]
and,  from  \cite[Eq. (24)]{ABGMS2015}, it follows that $D_x\mathcal{L}(t,u(t),u'(t))\in L^1([0,T],\rr^n)$.
Consequently, from \cite[p. 6]{mawhin2010critical}
(note that $\wphi_T$  includes the periodic test functions) we obtain the absolutely continuity of $D_y\mathcal{L}(t,u(t),u'(t))$ and that the differential equations in \eqref{ProbPrin-gral} are satisfied. The strict convexity of $\mathcal{L}(t,x,y)$ with respect to $y$ and the $T$-periodicity with respect to $t$ imply the boundary conditions in  \eqref{ProbPrin-gral} (see \cite[Thm. 4.1]{ABGMS2015}).\qed

 \section{An example}\label{sec:examples}

 In this section we develop  an application of our main result
so that the reader can appreciate the innovations that brings.

The main novelty  of our work is that we obtain existence of minima of $I$ associated with  lagrangian functions $\mathcal{L}(t,x,y)$ that do not satisfy a power-like grow condition on $y$.

In fact, it is possible to apply Theorem \ref{coercitividad-r} to lagrangians $\mathcal{L}=\mathcal{L}(t,x,y)$ with  exponential grow on the variable $y$. For example, suppose that
  \[\mathcal{L}(t,x,y)=f(y)+F(t,x),\]
with $f:\rr^n\to\rr$ differentiable, strictly convex and $f(y)\geq e^{|y|}$. We define for $n\geq 1$
\[\Phi(y)=e^y-\sum\limits_{i=0}^{n-1}\frac{y^i}{i!}.  \]
It is easy to see that $\Phi:[0,+\infty)\to [0,+\infty)$  is an $N$-function.
From \cite[Ex. 3, p. 85]{M} we know that $\alpha_{\Phi}=n$.
As consequence of \eqref{compl_ind} we have $\beta_{\Psi}=\frac{n}{n-1}<\infty$ and consequently $\Psi\in\Delta_2$.
From \eqref{delta2-potencias}, for every $1<p<n$ there exists $C_p>0$ such that
\[\Phi(rs)\geq C_pr^p\Phi(s),\quad r>1,s>0.\]
Then, the complementary pair $(\Phi,\Psi)$ and the $N$-function $\Phi^*(r):=r^p$ satisfy Lemma \ref{lem_coer} for every $1<p<n$.
Now, we fix arbitrary real numbers $1<p_0<p_1<p<n$ and we consider $\Phi_i=r^{p_i}$, $i=0,1$.
Then $\Phi_0\llcurly \Phi_1\llcurly \Phi^*$.
The conditions  \eqref{holder_cont-mu} and     \eqref{eq:propiedad-coercividad-phi0} become

\begin{equation}\label{eq:A6_ejem1}
  \left| \nabla F(t,x) \right|\leq b_1(t)|x|^{p_0-1}+b_2(t), \quad b_1,b_2\in L^1([0,T]),
\end{equation}
and
\begin{equation}\label{eq:ejem_propiedad-coercividad-phi0}
\lim_{|x|\to\infty}\frac{\int_{0}^{T}F(t,x)\ dt}{|x|^{(p_0-1)q_1 }}=+\infty,\quad q_1=p_1/(p_1-1),
\end{equation}
respectively. Since $n$ is an arbitrary positive integer, the pair $p_0$ and $p_1$ of real numbers with $1<p_0<p_1$ is also arbitrary.
For clarity,  assume that $F(t,x)=b(t)|x|^{\sigma}$, for some $1<\sigma<\infty$ and $b\in L^1([0,T])$.
We note that this $F$ satisfies \ref{item:condicion_a} and \ref{item:condicion_c}.
Now, we choose any $1<p_0$ with $p_0-1<\sigma<p_0$ and we take $p_1$ with $p_1>\sigma(\sigma-p_0+1)^{-1}$.
Then, \eqref{eq:A6_ejem1}
 and \eqref {eq:ejem_propiedad-coercividad-phi0} hold.
In conclusion, the action integral $I$ associated with the Lagrangian $\mathcal{L}(t,x,y)=f(y)+b(t)|x|^{\sigma}$   has minimum for any $1<\sigma$.

\def\cprime{$'$}

\end{document}